\documentclass[11pt,a4paper]{article}
\usepackage{amssymb}
\usepackage{eurosym}
\usepackage{amsfonts}
\usepackage{amsmath}
\usepackage{amsthm}
\usepackage{graphicx}
\usepackage{float}
\usepackage{hyperref}
\usepackage[pagewise]{lineno}

\hypersetup{
	colorlinks=true,
	linkcolor=blue,
	anchorcolor=blue,
	citecolor=blue
}
\newtheorem{theorem}{Theorem}[section]

\newtheorem{corollary}[theorem]{Corollary}
\newtheorem{lemma}[theorem]{Lemma}

\theoremstyle{definition}
\newtheorem{definition}[theorem]{Definition}

\newtheorem{remark}[theorem]{Remark}

\newtheorem*{da}{Data availability}
\renewenvironment{proof}[1][Proof]{\noindent\textbf{#1.} }{\ \rule{0.5em}{0.5em}}

\renewcommand{\theequation}{\thesection.\arabic{equation}}
\allowdisplaybreaks
\input{tcilatex}

\def\func#1{\mathop{\mathrm{#1}}\nolimits}

\def\Xint#1{\mathchoice
{\XXint\displaystyle\textstyle{#1}}%
{\XXint\textstyle\scriptstyle{#1}}%
{\XXint\scriptstyle\scriptscriptstyle{#1}}%
{\XXint\scriptscriptstyle\scriptscriptstyle{#1}}%
\!\int}
\def\XXint#1#2#3{{\setbox0=\hbox{$#1{#2#3}{\int}$ }
\vcenter{\hbox{$#2#3$ }}\kern-.6\wd0}}

\def\oint{\Xint-}

\def\enddoc{\end{document}}

\def\FRAME#1#2#3#4#5#6#7#8
{
 \begin{figure}[H]
 \begin{center}
 \includegraphics[height=#3]{#7}
 \caption{#5}
 \label{#6}
 \end{center}
 \end{figure}
}

\begin{document}
	\title{Long time upper bounds for solutions of Leibenson’s equation on Riemannian manifolds}
	\author{Philipp S\"urig}
	\date{April 2026}
	\maketitle
	
	\begin{abstract}
		We consider on Riemannian manifolds the Leibenson equation $$\partial _{t}u=\Delta _{p}u^{q}.$$ We prove that a certain upper bound for weak solutions of this equation is equivalent to a euclidean-type Sobolev inequality.
		
	\end{abstract}
	
	\let\thefootnote\relax\footnotetext{
		The author was funded by the Deutsche Forschungsgemeinschaft (DFG,
		German Research Foundation) - Project-ID 317210226 - SFB 1283.}
	
	\tableofcontents
	
	\section{Introduction}
	
	Let $M$ be a Riemannian manifold. We consider solutions of the non-linear evolution
	equation 
	\begin{equation}
		\partial _{t}u=\Delta _{p}u^{q},  \label{evoeq}
	\end{equation}%
	where $p>1, q>0,$ $u=u(x,t)$ is an unknown non-negative function of $x\in M$, $t\geq0$ and $%
	\Delta _{p}$ is the Riemannian $p$-Laplacian 
$\Delta _{p}v=\func{div}\left( |\nabla v|^{p-2}\nabla v\right).$
For the physical meaning of (\ref{evoeq}) see \cite{grigor2024finite, leibenzon1945general, leibenson1945turbulent}. 

The equation (\ref{evoeq}) is also referred to as \textit{Leibenson equation} or a \textit{doubly non-linear parabolic equation}. In the case $p=2$, it becomes a \textit{porous medium} equation $\partial _{t}u=\Delta u^{q}$,
and if in addition \(q=1\) then it amounts to the classical \textit{heat equation} $\partial _{t}u=\Delta u$.

Let $M$ be a geodesically complete Riemannian manifold of dimension $n$ and denote by $\mu$ the \textit{Riemannian measure} on $M$.
It was proved by Varopoulos \cite{varopoulos1985hardy} (see also \cite{grigor2006heat}) that, provided $n>2$, the Sobolev inequality 
\begin{equation}\label{weightsobointp2}\left(\int_{M}{|v|^{\frac{2n}{n-2}} d\mu}\right)^{\frac{n-2}{n}}\leq C\int_{M}{|\nabla v|^{2}d\mu} \quad \textnormal{for all}~v\in W^{1, 2}(M),\end{equation} is equivalent to the following upper bound for the solution $u$ of the heat equation with initial function $u_{0}=u(\cdot, 0)$: \begin{equation}\label{heatkernupp}||u(t)||_{L^{\infty}(M)}\leq ||u_{0}||_{L^{1}(M)}t^{-\frac{n}{2}}.\end{equation}

In the present paper, we prove a similar equivalence for solutions of the Leibenson equation (\ref{evoeq}).
Let us denote \begin{equation}\label{Dint}D=1-q(p-1).\end{equation}

The main result of the present paper is as follows  (cf. Theorem \ref{thmfinexinf} and Theorem \ref{sobolev}). 	

\begin{theorem}\label{mainthmint}
Assume that \begin{equation}\label{beta}p>nD.\end{equation}	
Then the Sobolev inequality \begin{equation}\label{weightsoboint}\left(\int_{M}{|v|^{\frac{pn}{n-p}} d\mu}\right)^{\frac{n-p}{n}}\leq C\int_{M}{|\nabla v|^{p}d\mu} \quad \textnormal{for all}~v\in W^{1, p}(M),\end{equation}	
assuming that $n>p$, is equivalent to the following estimate: for any non-negative bounded solution of (\ref{evoeq}) in $M\times \mathbb{R}_{+}$ with $u(\cdot, 0)=u_{0}\in L^{1}(M)\cap L^{\infty}(M)$ and, for  any $t>0$, \begin{equation}\label{decayofinf}||u(t)||_{L^{\infty}(M)}\leq C ||u_{0}||_{L^{1}(M)}^{\frac{p}{p-nD}}t^{-\frac{n}{p-nD}},\end{equation}where $C=C(p, q, n)$.
\end{theorem}
	
For example, inequality (\ref{weightsoboint}) holds if $M$ is a \textit{Cartan-Hadamard manifold} \cite{hoffman1974sobolev}, that is, $M$ is simply connected and has everywhere non-positive sectional curvature.

Let us recall some previously known results. 

In the case $q=1$, that is, (\ref{evoeq}) becomes the parabolic \textit{p-Laplace} equation, Theorem \ref{mainthmint} was proved by M. Bonforte and G. Grillo \cite{bonforte2007singular}.

When $p=2$, that is, (\ref{evoeq}) becomes the porous medium equation, Theorem \ref{mainthmint} was proved by M. Bonforte, G. Grillo, and J. L. Vazquez \cite{bonforte2008fast}.

Hence, the novelty of our Theorem \ref{mainthmint} is that it applies to the full range of $p>1$ and $q>0$ satisfying (\ref{beta}).

Let us discuss the upper bound (\ref{decayofinf}) in the special case when $M=\mathbb{R}^{n}$. In this case G. I. Barenblatt \cite{barenblatt1952self}
constructed for all $p>1$, $q>0$ spherically symmetric self-similar solutions of (\ref{evoeq}), that are nowadays called \textit{Barenblatt solutions}.
If (\ref{beta}) holds then the Barenblatt solution satisfies the estimate \begin{equation}\label{upperBar}||u(t)||_{L^{\infty}(\mathbb{R}^{n})}\leq C ||u_{0}||_{L^{1}(\mathbb{R}^{n})}^{\frac{p}{p-nD}}t^{-\frac{n}{p-nD}},\end{equation} which matches our estimate (\ref{decayofinf}) (see also \cite{bonforte2006super}). 

For further upper bounds for solutions of (\ref{evoeq}) on Riemannian manifolds in the case $D\geq 0$ we refer to \cite{grigor2026sharp, Grigoryan2024a, surig2024sharp}.

Let us also make a comment on the proof of the upper bound (\ref{decayofinf}) in Theorem \ref{mainthmint}. The proof utilizes a \textit{$\log$-Sobolev} inequality, which follows from the Sobolev inequality (\ref{weightsoboint}) (cf. Lemma \ref{logsobolev}). This method is inspired by \cite{bonforte2007singular, bonforte2008fast} (see also \cite{cipriani2001uniform}). 

The structure of this paper is as follows.

In Section \ref{secweak}, we
define the notion of a weak solution of the Leibenson equation (\ref{evoeq}).

In Section \ref{Secini} we prove in Theorem \ref{thmfinexinf} that the Sobolev inequality (\ref{weightsoboint}) implies the upper bound (\ref{decayofinf}).

In Section \ref{secsobo} we prove in Theorem \ref{sobolev} the opposite direction, that is, the upper bound (\ref{decayofinf}) implies the Sobolev inequality (\ref{weightsoboint}).

For other qualitative and quantitative properties of solutions of (\ref{evoeq}) on Riemannian manifolds see \cite{meglioli2025global, surig2024finite, surig2025gradient} and from a probabilistic point of view \cite{barbu2025leibenson}.
	
	\section{Weak solutions}
	
	\label{secweak}
	
	We consider in what follows the following non-linear evolution
	equation on a Riemannian manifold $M$:%
	\begin{equation}
		\partial _{t}u=\Delta _{p}u^{q}.  \label{dtv}
	\end{equation} 
	By a \textit{solution} of (\ref{dtv}) we mean a non-negative function $u$
	satisfying (\ref{dtv})
	in a certain weak sense as explained below.
	
	We assume throughout that 
	\begin{equation*}
		p>1\ \ \text{and}\ \ \ q>0.
	\end{equation*}%
	Set%
	\begin{equation}\label{D}
		D =1-q(p-1).
	\end{equation}%
Let $I$ be an interval in $\mathbb{R}_{+}=[0, \infty)$.
	\begin{definition}
		\normalfont
		We say that a non-negative function $u=u(x, t)$ is a \textit{weak
			solution} of (\ref{dtv}) in $M\times I$, if 
		\begin{equation}  \label{defvonsoluq}
			u\in C\left(I; L^{1+q}(M)\right)\quad \textnormal{and}\quad 
			u^{q}\in L_{loc}^{p}\left(I; W^{1, p}(M)\right),
		\end{equation} such that $$\partial_{t}u^{\frac{1+q}{2}}\in L^{2}_{loc}(I, L^{2}(M)),$$ where (\ref{dtv}) holds weakly in $M\times I$, which means that for all $t_{1}, t_{2}\in I$ with $t_{1}<t_{2}$, and all \textit{test functions} 
		\begin{equation}  \label{defvontestsoluq}
			\psi\in W_{loc}^{1, 1+\frac{1}{q}}\left(I;
			L^{1+\frac{1}{q}}(M)\right)\cap L_{loc}^{p}\left(I; W^{1,
				p}(M)\right),
		\end{equation}
		we have 
		\begin{equation}  \label{defvonweaksolq}
			\left[\int_{M}{u\psi}\right]_{t_{1}}^{t_{2}}+\int_{t_{1}}^{t_{2}}{%
	\int_{M}{-u\partial_{t}\psi+|\nabla u^{q}|^{p-2}\langle\nabla u^{q},
		\nabla \psi\rangle}}= 0.
		\end{equation}
	\end{definition} 

Existence results for weak solutions of (\ref{dtv}) were obtained in \cite{alt1983quasilinear,andreucci1990new, benilan1995strong, coulhon2016regularisation, ishige1996existence, ivanov1997regularity, raviart1970resolution,  tsutsumi1988solutions} in the euclidean setting and in \cite{andreucci2015optimal, de2022wasserstein, surig2026existence} on Riemannian manifolds.
	
	The next two lemma can be proved similarly to Lemma 2.6 in \cite{grigor2024finite}.
	\begin{lemma}[Caccioppoli type inequality]
		\label{Lem1} Let $u=u\left( x,t\right) $ be a bounded
		non-negative solution to \emph{(\ref{dtv})} in $M\times I$. Fix some real $\lambda $
		such that 
		\begin{equation}
			\lambda \geq 1+q  \label{la>3-m}
		\end{equation}%
	and set $\sigma=\lambda-D.$
		Choose $t_{1},t_{2}\in I$ such that $t_{1}<t_{2}$. Then%
		\begin{equation}
			\left[ \int_{M }u^{\lambda } d\mu\right] _{t_{1}}^{t_{2}}+c_{1}\lambda\left( \lambda  -1\right)\sigma ^{-p}
			\int_{M\times [t_{1}, t_{2}]}\left\vert \nabla \left( u^{\sigma/p } \right) \right\vert
			^{p}d\mu dt= 0,
			\label{vetacor}
		\end{equation}
		where $c_{1}$ is a constant depending on $p$ and $q$. 
	\end{lemma}
	
	Let us observe for a later usage that \begin{equation}\label{valpha}u^{\sigma/p}\in L^{p}\left(I; W^{1, p}(M)\right).\end{equation}
	Indeed, using $\sigma/p\geq q$, we get that the function $\Phi(s)=s^{\frac{\sigma}{pq}}$ is Lipschitz on any bounded interval in $[0, \infty)$. Thus, $u^{\sigma/p}=\Phi(u^{q})\in W^{1, p}(M)$ and $\left|\nabla u^{\sigma/p}\right|=\left|\Phi^{\prime}(u^{q})\nabla u^{q}\right|\leq C\left|\nabla u^{q}\right|,$ whence for any bounded interval $J\subset I$, \begin{equation*}\label{valpha0}\int_{M\times J}u^{\sigma}+\left\vert \nabla \left( u^{\sigma/p } \right) \right\vert^{p}\leq  C^{\prime}\int_{M\times J}u^{\sigma}+\left\vert \nabla  u^{q } \right\vert^{p},\end{equation*} which is finite since $$\int_{M\times J}u^{\sigma}\leq \textnormal{const}~||u||_{L^{\infty}}^{\sigma-pq}\int_{M\times J}u^{pq}$$ and proves (\ref{valpha}).

	\begin{lemma}[Lemma 2.9 \cite{Grigoryan2024}]
		\label{monl1}
		Let $u=u\left( x,t\right) $ be a non-negative bounded solution to \emph{(\ref{dtv})} in $M\times I$. If $\lambda\geq 1$, including $\lambda=\infty$, then the function 
		\begin{equation*}
			t\mapsto \left\Vert u(\cdot ,t)\right\Vert _{L^{\lambda}(M)}
		\end{equation*}%
		is non-increasing in $I$.
	\end{lemma}

\begin{lemma}
	\label{nabdecreasing}
	Let $u$ be a non-negative bounded solution of (\ref{dtv}) in $M\times (0, T)$. Then, for almost any $t\in (0, T)$,
	\begin{equation*}
	\int_{M}\left\vert \nabla  u^{q } \right\vert^{p}d\mu (t)\leq \int_{M}\left\vert \nabla  u_{0}^{q } \right\vert^{p}d\mu.
	\end{equation*}%
\end{lemma}

\begin{proof}
Since $u$ is a weak solution
of (\ref{dtv}) we obtain by testing in (\ref{defvonweaksolq}) with test function $$\psi(x, s)=\theta_{\nu}(s)\frac{u^{q}(x, s+h)-u^{q}(s)}{h},$$ where $h, \nu>0$ and \begin{equation*}
	\theta_{\nu}(s)=\left\{ 
	\begin{array}{ll}
		0, & s\leq 0,\\
		\frac{1}{\nu}s, & 0<s<\nu, \\ 
		1, & \nu\leq s<t-\nu, \\ 
		\frac{1}{\nu}(t-s), & t-\nu\leq s<t, \\ 
		0, & s\geq t,%
	\end{array}%
	\right.
\end{equation*} which is admissible since $u^{q}(\cdot, s)\in W^{1, p}(M)$ for almost any $s\in (0, T)$, $$\int_{Q}-u\partial_{s}\psi+|\nabla u^{q}|^{p-2}\langle\nabla u^{q},
\nabla \psi\rangle d\mu ds=0,$$ where $Q=M\times [0, t]$. Here we have used that $\psi(\cdot, 0)=\psi(\cdot, t)=0$. Further we can compute \begin{align*}
\int_{Q}|\nabla u^{q}|^{p-2}\langle\nabla u^{q},
\nabla \psi\rangle d\mu ds&=\frac{1}{h}\int_{Q}\theta_{\nu}(s)|\nabla u^{q}|^{p-2}\langle\nabla u^{q},
\nabla u^{q}(x, s+h)-\nabla u^{q}(x, s) \rangle d\mu ds\\&=-\frac{1}{h}\int_{Q}(\theta_{\nu}(s)-\theta_{\nu}(s-h))|\nabla u^{q}|^{p}
d\mu ds,
\end{align*} and $$\frac{1}{h}\int_{Q}(\theta_{\nu}(s)-\theta_{\nu}(s-h))|\nabla u^{q}|^{p}
d\mu ds\to \int_{Q}\theta_{\nu}^{\prime}(s)|\nabla u^{q}|^{p}
d\mu ds\quad \textnormal{as}~h\to 0.$$ On the other hand, $$\int_{Q}-u\partial_{s}\psi d\mu ds=\frac{1}{h}\int_{Q}\theta_{\nu}\partial_{s}u(u^{q}(x, s+h)-u^{q}(s))d\mu ds$$ and $$\frac{1}{h}\int_{Q}\theta_{\nu}\partial_{s}u(u^{q}(x, s+h)-u^{q}(s))d\mu ds\to \int_{Q}\theta_{\nu}\partial_{s}u\partial_{s}u^{q}d\mu ds\quad \textnormal{as}~h\to 0.$$ Since $$\int_{Q}\theta_{\nu}\partial_{s}u\partial_{s}u^{q}d\mu ds=\frac{4q}{(1+q)^{2}}\int_{Q}\theta_{\nu}(\partial_{s}u^{\frac{1+q}{2}})^{2}d\mu ds\geq0, $$ we get $$\int_{Q}\theta_{\nu}^{\prime}(s)|\nabla u^{q}|^{p}
d\mu ds\geq 0.$$
Hence, sending $\nu\to 0$, we conclude for almost any $t\in (0, T)$, $$\left[\int_{M}|\nabla u^{q}|^{p}d\mu \right]_{0}^{t}\leq0,$$ which proves the claim.
\end{proof}

Define, for any non-negative function $v$ and any $r>0$, the functional \begin{equation}\label{funcsobo}J(r, v)=\int_{M}\frac{v^{r}}{||v||_{L^{r}}^{r}}\log\left(\frac{v}{||v||_{L^{r}}}\right)d\mu.\end{equation}

\begin{lemma}\label{difffunc}
	Let $u$ be a non-negative bounded solution of (\ref{dtv}) in $M\times (0, t)$. Let $\lambda(s)$ be a continuously differentiable non-decreasing function on $[0, t)$ such that $\lambda(s)\geq 1+q$ and set $\sigma(s)=\lambda(s)-D$. Then, for almost every $s\in(0, t)$, \begin{equation}\label{logdiff}
		\frac{d}{ds}\log ||u(s)||_{L^{\lambda(s)}}\leq\frac{\lambda^{\prime}(s)}{\lambda(s)}J(\lambda(s), u(s))-\frac{c_{1}(\lambda(s)-1)}{\sigma(s)^{p}}\frac{||\nabla \left( u^{\sigma(s)/p } \right)||_{L^{p}}^{p}}{||u(s)||_{L^{\lambda(s)}}^{\lambda(s)}}.
	\end{equation}
\end{lemma}

\begin{proof}
	Let us set $$\Phi(s, \lambda(s))=\int_{M }u^{\lambda(s) } d\mu(s).$$ It follows from (\ref{vetacor}) that, for almost any $s\in(0, t)$,
	$$\Phi(s+h, \lambda(s))-\Phi(s, \lambda(s))=-c_{1}\lambda(s)\left( \lambda (s) -1\right)\sigma(s) ^{-p}
	\int_{s}^{s+h}\int_{M}\left\vert \nabla \left( u^{\sigma(s)/p } \right) \right\vert
	^{p}d\mu dt,$$ whence dividing by $h$ and sending $h\to 0$, we get, 
	$$\partial_{s}\Phi(s, \lambda(s))=-c_{1}\lambda(s)\left( \lambda(s)  -1\right)\sigma(s) ^{-p}
	\int_{M}\left\vert \nabla \left( u^{\sigma(s)/p } \right) \right\vert^{p}d\mu.$$
	Thus, we obtain \begin{align*}
		\frac{d}{ds} \Phi(s, \lambda(s))&=\partial_{s}\Phi(s, \lambda(s))+\lambda^{\prime}(s)\partial_{\lambda}\Phi(s, \lambda(s))\\&=-c_{1}\lambda(s)\left( \lambda(s)  -1\right)\sigma(s) ^{-p}
		\int_{M}\left\vert \nabla \left( u^{\sigma(s)/p } \right) \right\vert^{p}d\mu +\lambda^{\prime}(s)\int_{M}\log u u^{\lambda}.
	\end{align*}
	Therefore, \begin{align*}\frac{d}{ds} \log||u(s)||_{L^{\lambda(s)}}&=\frac{d}{ds}\left(\frac{1}{\lambda(s)}\log\Phi(s, \lambda(s))\right)\\&=-\frac{\lambda^{\prime}(s)}{\lambda(s)}\log||u(s)||_{L^{\lambda(s)}}+\frac{1}{\lambda(s)\Phi(s, \lambda(s))}\frac{d}{ds} \Phi(s, \lambda(s))\\&=-\frac{\lambda^{\prime}(s)}{\lambda(s)}\log||u(s)||_{L^{\lambda(s)}}-\frac{c_{1}\left( \lambda(s)  -1\right)}{\sigma(s)^{p}\Phi(s, \lambda(s))}\int_{M}\left\vert \nabla \left( u^{\sigma(s)/p } \right) \right\vert^{p}d\mu\\&\quad +\frac{\lambda^{\prime}(s)}{\lambda(s)\Phi(s, \lambda(s))}\int_{M}\log u u^{\lambda}.
	\end{align*}
	Hence, we obtain $$\frac{d}{ds} \log||u(s)||_{L^{\lambda(s)}}\leq \frac{\lambda^{\prime}(s)}{\lambda(s)}\int_{M}\frac{u^{\lambda(s)}}{||u||_{L^{\lambda(s)}}}\log\left(\frac{u}{||u||_{L^{\lambda(s)}}}\right)d\mu-\frac{c_{1}\left( \lambda(s)  -1\right)}{\sigma(s)^{p}\Phi(s, \lambda(s))}\int_{M}\left\vert \nabla \left( u^{\sigma(s)/p } \right) \right\vert^{p}d\mu ,$$ which implies (\ref{logdiff}) and finishes the proof.
\end{proof}
	
	\section{Upper bounds}
	
	\label{Secini} 
	\begin{definition}
		We say that $M$ satisfies a Sobolev inequality if there exists a constant $S_{M}>0$ such that \begin{equation}\label{weightsobo}\left(\int_{M}{|v|^{p\kappa} d\mu}\right)^{1/\kappa}\leq S_{M}\int_{M}{|\nabla v|^{p}d\mu} \quad \textnormal{for all}~v\in W^{1, p}(M),\end{equation} where $\kappa>1$.
	\end{definition} 
Let us also set \begin{equation}\label{nu}\nu=\frac{\kappa-1}{\kappa}.\end{equation} 
When $M$ is a Cartan-Hadamard manifold with $\textnormal{dim}~M=n>p$, (\ref{weightsobo}) holds with $\kappa=\frac{n}{n-p}$ and $S_{M}\leq \textnormal{const}$ and we have $\nu=\frac{p}{n}$.

\begin{lemma}\label{logsobolev}
Suppose that $M$ satisfies the Sobolev inequality (\ref{weightsobo}). Let $J$ be defined by (\ref{funcsobo}). Then, for any $\varepsilon>0$, any $r\in (0, p\kappa)$ and any non-negative $v\in  W^{1, p}(M)$, we have \begin{equation}\label{logsobolevin}rJ(r,v)\leq \frac{r\kappa S_{M}\varepsilon}{p\kappa-r}\frac{||\nabla v||_{L^{p}}^{p}}{||v||_{L^{r}}^{p}}-\frac{r\kappa}{p\kappa-r}\log \varepsilon.\end{equation}
\end{lemma}

\begin{proof}
Setting $\alpha=p\kappa-r>0$ we obtain by Jensen's inequality with respect to the probability measure $\frac{v^{r}}{||v||_{L^{r}}^{r}}d\mu$ and the Sobolev inequality (\ref{weightsobo}) \begin{align*}
rJ(r,v)&=\frac{r}{\alpha}\int_{M}\frac{v^{r}}{||v||_{L^{r}}^{r}}\log\left(\frac{v^{\alpha}}{||v||_{L^{r}}^{\alpha}}\right)d\mu\\&\leq \frac{r}{\alpha}\log\left(\frac{||v||_{L^{\alpha+r}}^{\alpha+r}}{||v||_{L^{r}}^{\alpha+r}}\right)=\frac{rp\kappa}{p\kappa-r}\log\left(\frac{||v||_{L^{p\kappa}}}{||v||_{L^{r}}}\right)\\&\leq \frac{rp\kappa}{p\kappa-r}\log\left(\frac{S_{M}^{1/p}||\nabla v||_{L^{p}}}{||v||_{L^{r}}}\right).
\end{align*}
Since $\log x\leq \varepsilon x-\log \varepsilon$, for all positive $x,\varepsilon$, it follows that $$rJ(r,v)\leq \frac{r\kappa S_{M}\varepsilon}{p\kappa-r}\frac{||\nabla v||_{L^{p}}^{p}}{||v||_{L^{r}}^{p}}-\frac{r\kappa}{p\kappa-r}\log \varepsilon,$$ which is (\ref{logsobolevin}) and finishes the proof.
\end{proof}

	The following theorem is the first part our main result Theorem \ref{mainthmint}.  
	
	\begin{theorem}\label{thmfinexinf}
	Let $u$ be a non-negative bounded solution of (\ref{dtv}) in $M\times \mathbb{R}_{+}$ with $u(\cdot, 0)=u_{0}\in L^{1}(M)\cap L^{\infty}(M)$. Suppose that $M$ satisfies the Sobolev inequality (\ref{weightsobo}). Then, for any $a\geq 1$ satisfying \begin{equation}\label{ainf}a>\frac{D}{\nu} ,\end{equation} we have, for all $t>0$, \begin{equation}\label{decayofinfthm}||u(t)||_{L^{\infty}(M)}\leq C ||u_{0}||_{L^{a}(M)}^{\frac{a\nu}{a\nu-D}}t^{-\frac{1}{a\nu-D}},\end{equation} where $C$ depends on $p, q, \kappa, a$ and $S_{M}$.
\end{theorem}

\begin{proof}
Let $\lambda(s)$ and $\sigma(s)$ be as in Lemma \ref{difffunc} such that $\lambda(0)=a$ and $\lambda(s)\to +\infty$ as $s\to t$. Then we have from (\ref{logdiff}) that $$\frac{d}{ds}\log ||u(s)||_{L^{\lambda(s)}}\leq\frac{\lambda^{\prime}(s)}{\lambda(s)}J(\lambda(s), u(s))-\frac{c_{1}(\lambda(s)-1)}{\sigma(s)^{p}}\frac{||\nabla \left( u^{\sigma(s)/p } \right)||_{L^{p}}^{p}}{||u(s)||_{L^{\lambda(s)}}^{\lambda(s)}}.$$
With $v=u^{\sigma(s)/p}\in W^{1, p}(M)$ (see (\ref{valpha})) it follows from (\ref{logsobolevin}) that, for any $\varepsilon>0$ and any $r\in (0, p\kappa)$,
$$||\nabla u^{\sigma(s)/p}||_{L^{p}}^{p}\geq \frac{(p\kappa-r)||u^{\sigma(s)/p}||_{L^{r}}^{p}}{r\kappa S_{M}\varepsilon}\left(rJ\left(r, u^{\sigma(s)/p}\right)+\frac{r\kappa}{p\kappa-r}\log \varepsilon\right).$$
Combining these inequalities, we deduce \begin{align*}\frac{d}{ds}\log ||u(s)||_{L^{\lambda(s)}}\leq \frac{\lambda^{\prime}(s)}{\lambda(s)}J(\lambda(s), u(s))&-\frac{c_{1}(\lambda(s)-1)(p\kappa-r)||u^{\sigma(s)/p}||_{L^{r}}^{p}}{\sigma(s)^{p}||u(s)||_{L^{\lambda(s)}}^{\lambda(s)}r\kappa S_{M}\varepsilon}\\&\times\left(rJ\left(r, u^{\sigma(s)/p}\right)+\frac{r\kappa}{p\kappa-r}\log \varepsilon\right).\end{align*}
Let us choose \begin{equation}\label{epsilon}\varepsilon=\frac{\lambda(s)^{2}}{\lambda^{\prime}(s)}\frac{c_{1}(\lambda(s)-1)(p\kappa-r)}{\sigma(s)^{p}r\kappa S_{M}}\frac{||u^{\sigma(s)/p}||_{L^{r}}^{p}}{||u(s)||_{L^{\lambda(s)}}^{\lambda(s)}}=:\varepsilon_{1}\frac{||u^{\sigma(s)/p}||_{L^{r}}^{p}}{||u(s)||_{L^{\lambda(s)}}^{\lambda(s)}}.\end{equation} For such $\varepsilon$ we obtain $$\frac{d}{ds}\log ||u(s)||_{L^{\lambda(s)}}\leq \frac{\lambda^{\prime}(s)}{\lambda(s)^{2}}\left[\lambda(s)J(\lambda(s), u(s))-rJ\left(r, u^{\sigma(s)/p}\right)\right]-\frac{\lambda^{\prime}(s)}{\lambda(s)^{2}}\frac{r\kappa}{p\kappa-r}\log \varepsilon.$$
Since $$\lambda(s)J(\lambda(s), u(s))-rJ\left(r, u^{\sigma(s)/p}\right)=J\left(1, u(s)^{\lambda(s)}\right)-J\left(1, u^{(\sigma(s)r)/p}\right),$$ we obtain that this term is zero if $$r=\frac{\lambda(s)p}{\sigma(s)}.$$
Note that $r<p\kappa$ is then equivalent to $\lambda(s)>D\frac{\kappa}{\kappa-1}=\frac{D}{\nu}$, which holds true because, by our assumption (\ref{ainf}), $\lambda(0)=a>\frac{D}{\nu}$ and since $\lambda$ is non-decreasing. 
With this choice of $r$ we have $$\frac{r\kappa}{p\kappa-r}=\frac{\lambda(s)}{\lambda(s)\nu-D}$$ and thus \begin{align*}\frac{d}{ds}\log ||u(s)||_{L^{\lambda(s)}}&\leq -\frac{\lambda^{\prime}(s)}{\lambda(s)}\frac{1}{\lambda(s)\nu-D}\log \varepsilon\\&=-\frac{\lambda^{\prime}(s)}{\lambda(s)}\frac{1}{\lambda(s)\nu-D}\left(\log \varepsilon_{1}+\log\left(\frac{||u||_{L^{\lambda(s)}}^{\sigma(s)}}{||u||_{L^{\lambda(s)}}^{\lambda(s)}}\right)\right)\\&=-\frac{\lambda^{\prime}(s)}{\lambda(s)}\frac{1}{\lambda(s)\nu-D}\left(\log \varepsilon_{1}-D\log ||u||_{L^{\lambda(s)}}\right).\end{align*}
Hence, \begin{align}\nonumber\frac{d}{ds}\log ||u(s)||_{L^{\lambda(s)}}\leq \frac{\lambda^{\prime}(s)}{\lambda(s)}&\frac{D}{\lambda(s)\nu-D}\log ||u||_{L^{\lambda(s)}}\\&\label{diffforlog}-\frac{\lambda^{\prime}(s)}{\lambda(s)}\frac{1}{\lambda(s)\nu-D}\log\left(\frac{\lambda(s)}{\lambda^{\prime}(s)}\frac{c_{1}(\lambda(s)-1)(\lambda(s)\nu-D)}{(\lambda(s)-D)^{p} S_{M}}\right).
\end{align}
Therefore, setting $\Phi(s)=\log ||u(s)||_{L^{\lambda(s)}}$, we can write (\ref{diffforlog}) as $$\frac{d}{ds}\Phi(s)+f(s)\Phi(s)+g(s)\leq 0,$$ where $$f(s)=-\frac{\lambda^{\prime}(s)}{\lambda(s)}\frac{D}{\lambda(s)\nu-D}\quad \textnormal{and}\quad g(s)=\frac{\lambda^{\prime}(s)}{\lambda(s)}\frac{1}{\lambda(s)\nu-D}\log\left(\frac{\lambda(s)}{\lambda^{\prime}(s)}\frac{c_{1}(\lambda(s)-1)(\lambda(s)\nu-D)}{(\lambda(s)-D)^{p} S_{M}}\right).$$
The ODE $$\frac{d}{ds}\Psi(s)+f(s)\Psi(s)+g(s)= 0$$ has an explicit solution of the form $$\Psi(s)=\exp\left(-\int_{0}^{s}f(\omega)d\omega\right)\left[\Psi(0)-\int_{0}^{s}g(\omega)\exp\left(\int_{0}^{\omega}f(\eta)d\eta\right)d\omega\right].$$
Using the substitution $\xi=\lambda(\omega)$ we have \begin{align*}\int_{0}^{s}f(\omega)d\omega&=-\int_{0}^{s}\frac{\lambda^{\prime}(\omega)}{\lambda(\omega)}\frac{D}{\lambda(\omega)\nu-D}d\omega=-\int_{\lambda(0)}^{\lambda(s)}\frac{D}{\xi(\xi\nu-D)}d\xi\\&=\left[\log\frac{\xi}{\xi\nu-D}\right]_{a}^{\lambda(s)}=\log\left(\frac{\lambda(s)(a\nu-D)}{a(\lambda(s)\nu-D)}\right).
\end{align*}
Since $\lambda(s)\to +\infty$ as $s\to t$, we get that $$\lim_{s\to t}\int_{0}^{s}f(\omega)d\omega= \log\left(\frac{a\nu-D}{a\nu}\right).$$
Further, we have \begin{align*}
\int_{0}^{s}g(\omega)&\exp\left(\int_{0}^{\omega}f(\eta)d\eta\right)d\omega\\&=\frac{a\nu-D}{a}\int_{0}^{s}\frac{\lambda^{\prime}(\omega)}{(\lambda(\omega)\nu-D)^{2}}\left[\log\left(\frac{\lambda(\omega)^{2}c_{1}}{\lambda^{\prime}(\omega)S_{\Omega}}\right)+\log\left(\frac{(\lambda(\omega)-1)(\lambda(\omega)\nu-D)}{\lambda(\omega)(\lambda(\omega)-D)^{p}}\right)\right]d\omega\\&=:\frac{a\nu-D}{a}(I_{1}(s)+I_{2}(s))
\end{align*}
Let us choose $\lambda(s)=\frac{at}{t-s}$. Using that $\lambda^{\prime}(s)=\frac{at}{(t-s)^{2}}$, we get since $\lambda(0)=a$, $$I_{1}(s)=\log\left(\frac{at c_{1}}{S_{M}}\right)\int_{0}^{s}\frac{\lambda(\omega)d\omega}{(\lambda(\omega)\nu-D)^{2}}=\log\left(\frac{ac_{1}t}{S_{M}}\right)\left(\frac{1}{\nu(a\nu-D)}-\frac{1}{\nu(\lambda(s)\nu-D)}\right).$$
Again, using the substitution $\xi=\lambda(s)$, we have $$I_{2}(s)=\int_{a}^{\lambda(s)}\frac{1}{(\xi\nu-D)^{2}}\log\left(\frac{(\xi-1)(\xi\nu-D)}{\xi(\xi-D)^{p}}\right)d\xi.$$ Note that the integrand has no singularities in the domain of integration by our assumption on $\lambda$ and $a$. Hence, the integral converges as $s\to t$, that is, as $\lambda(s)\to +\infty$.
Therefore, we deduce $$\int_{0}^{s}g(\omega)\exp\left(\int_{0}^{\omega}f(\eta)d\eta\right)d\omega=\frac{a\nu-D}{a}\log\left(t\right)\left(\frac{1}{\nu(a\nu-D)}-\frac{1}{\nu(\lambda(s)\nu-D)}\right)+G(s),$$ where $G(s)$ is a function depending $c_{1}, S_{M}, a, p, q, \nu$ that converges to some constant when $s\to t$.
Thus, \begin{align*}\Psi(s)&=\frac{a(\lambda(s)\nu-D)}{\lambda(s)(a\nu-D)}\Psi(0)-\frac{\lambda(s)\nu-D}{\lambda(s)}\log\left(t\right)\left(\frac{1}{\nu(a\nu-D)}-\frac{1}{\nu(\lambda(s)\nu-D)}\right)+\frac{a(\lambda(s)\nu-D)}{\lambda(s)(a\nu-D)}G(s)\\&=\frac{a(\lambda(s)\nu-D)}{\lambda(s)(a\nu-D)}\Psi(0)-\frac{\lambda(s)-a}{\lambda(s)(a\nu-D)}\log(t)+\frac{a(\lambda(s)\nu-D)}{\lambda(s)(a\nu-D)}G(s).\end{align*}
Sending $s\to t$, that is, $\lambda(s)\to +\infty$, we get $$\Psi(t)=\frac{a\nu}{a\nu-D}\Psi(0)-\frac{1}{a\nu-D}\log t+C,$$ where $C$ depends on $c_{1}, S_{M}, a, p, q, \nu$.
Hence, $$\log ||u(t)||_{L^{\infty}}=\lim_{s\to t}\log ||u(t)||_{L^{\lambda(s)}}\leq \lim_{s\to t}\log ||u(s)||_{L^{\lambda(s)}}=\lim_{s\to t}\Phi(s)\leq \lim_{s\to t}\Psi(s)=\Psi(t),$$ so that letting $\Psi(0)=\Phi(0)=\log||u_{0}||_{L^{a}}$ we conclude $$||u(t)||_{L^{\infty}}\leq e^{C}||u_{0}||_{L^{a}}^{\frac{a\nu}{a\nu-D}}t^{-\frac{1}{a\nu-D}},$$ which implies (\ref{decayofinfthm}) and finishes the proof.
\end{proof}

From interpolation we get the following result:

\begin{corollary}\label{corlambda}
Under the assumptions of Theorem \ref{thmfinexinf}, for any $\lambda\geq 1+q$ and $a\geq 1$ satisfying \begin{equation}\label{lambdanew}\lambda> a>\frac{D}{\nu} ,\end{equation} we have, for all $t>0$,
\begin{equation}\label{forequi}||u(t)||_{L^{\lambda}(M)}\leq C ||u_{0}||_{L^{a}(M)}^{\gamma}t^{-\alpha},\end{equation} where \begin{equation}\label{gammanew}\alpha=\frac{1}{\lambda}\frac{\lambda-a}{a\nu-D}\quad\textnormal{and}\quad\gamma=\frac{a}{\lambda}\frac{\lambda \nu-D}{a\nu-D}.\end{equation}
\end{corollary}

\begin{proof}
We have \begin{align*}
||u(t)||_{L^{\lambda}(M)}\leq ||u(t)||_{L^{\infty}(M)}^{1-\frac{a}{\lambda}}||u(t)||_{L^{a}(M)}^{\frac{a}{\lambda}}.
\end{align*}
Since $||u(t)||_{L^{a}(M)}\leq ||u_{0}||_{L^{a}(M)}$ by Lemma \ref{monl1}, we obtain from (\ref{decayofinfthm}) $$||u(t)||_{L^{\lambda}(M)}\leq C ||u_{0}||_{L^{a}(M)}^{a\frac{a\nu}{a\nu-D}(1-\frac{a}{\lambda})+\frac{a}{\lambda}}t^{-\frac{1}{a\nu-D}(1-\frac{a}{\lambda})},$$ which implies (\ref{forequi}).
\end{proof}

\begin{remark}
Let $M=\mathbb{R}^{n}$, $n>p$, so that $\kappa=\frac{n}{n-p}$, then $$||u(t)||_{L^{\infty}(\mathbb{R}^{n})}\leq C ||u_{0}||_{L^{a}(\mathbb{R}^{n})}^{\frac{ap}{ap-nD}}t^{-\frac{n}{ap-nD}}$$ and $$||u(t)||_{L^{\lambda}(\mathbb{R}^{n})}\leq C||u_{0}||_{L^{a}(\mathbb{R}^{n})}^{\frac{a(p\lambda-nD)}{\lambda(ap-nD)}}t^{-\frac{n(\lambda-a)}{\lambda(ap-nD)}},$$ which matches the estimate for the Barenblatt solution (\ref{upperBar}) and the optimal estimates obtained in \cite{bonforte2006super} in the case $D\leq 0$.
\end{remark}

\section{Sobolev inequalities}
\label{secsobo}

	\begin{theorem}\label{sobolev}
	Assume that \begin{equation}\label{condforsobo}1+q>\frac{D}{\nu},\end{equation} where $\nu$ is defined by (\ref{nu}) for some $\kappa>1$.
	Let $u$ be a non-negative bounded solution of (\ref{dtv}) in $M\times \mathbb{R}_{+}$ with $u(\cdot, 0)=u_{0}\in L^{1}(M)\cap L^{\infty}(M)$. Suppose that
	\begin{equation}\label{condforsoboinf}||u(t)||_{L^{\infty}(M)}\leq C ||u_{0}||_{L^{a}(M)}^{\frac{a\nu}{a\nu-D}}t^{-\frac{1}{a\nu-D}}
	\end{equation}	
	where $a\geq 1$ is such that  \begin{equation*}\frac{D}{\nu}<a<1+q.\end{equation*}
	Then the Sobolev inequality (\ref{weightsobo}) holds in $M$ with Sobolev exponent $\kappa$.
	\end{theorem}

\begin{remark}
Clearly, if $D<\nu$, condition (\ref{condforsobo}) is automatically satisfied. If $p=2$, $n>2$ and $\kappa=\frac{n}{n-2}$ condition (\ref{condforsobo}) becomes $q>\frac{n-2}{n+2}$.
\end{remark}

\begin{proof}
From Lemma \ref{Lem1} we have, for any $t>0$, \begin{equation*}
	\left[ \int_{M }u^{1+q } d\mu\right] _{0}^{t}+c_{1}%
	\int_{M\times [0, t]}\left\vert \nabla  u^{q } \right\vert
	^{p}d\mu dt=0.
\end{equation*}
Then we obtain by Lemma \ref{nabdecreasing} \begin{align*}
	\int_{M }u^{1+q } d\mu(t)\geq -c_{1}t
	\int_{M}\left\vert \nabla  u_{0}^{q } \right\vert
	^{p}d\mu +\int_{M }u_{0}^{1+q } d\mu.
\end{align*}
We can prove as in Corollary \ref{corlambda} that our assumption (\ref{condforsoboinf}) implies (\ref{forequi}).
Hence, we can apply (\ref{forequi}) with $\lambda=1+q$ and $\alpha$ and $\gamma$ given by (\ref{gammanew}) with $\lambda=1+q$  respectively that
\begin{equation}\label{decayofnormq+1}\int_{M }u^{1+q }(t) \leq C\left(\int_{M}u_{0}^{a}\right)^{\gamma(1+q)/a}t^{-\alpha(1+q)}.\end{equation}
Let us set $\beta:=\alpha(1+q)$ and $\zeta:=\gamma(1+q)/a$.
Then, we obtain from (\ref{decayofnormq+1}), $$C\left(\int_{M}u_{0}^{a}\right)^{\zeta}t^{-\beta}\geq -c_{1}t%
\int_{M}\left\vert \nabla  u_{0}^{q } \right\vert
^{p}d\mu  +\int_{M }u_{0}^{1+q } d\mu,$$ that is, $$f(t)=At^{-\beta}+Bt\geq c,$$ where $$A=C\left(\int_{M}u_{0}^{a}\right)^{\zeta}, \quad B=c_{1}%
\int_{M}\left\vert \nabla  u_{0}^{q } \right\vert
^{p}d\mu\quad \textnormal{and}\quad c=\int_{M }u_{0}^{1+q } d\mu.$$ Since $f^{\prime}(t)=-A\beta t^{-\beta -1}+B$, the function $f$ has a minimum at $$t_{0}=\left(\frac{A\beta}{B}\right)^{\frac{1}{\beta+1}},$$ and for this $t_{0}$ $$f(t_{0})=KA^{\frac{1}{\beta+1}}B^{\frac{\beta}{\beta+1}},$$ where $K$ is a positive constant. Therefore, we have $f(t)\geq f(t_{0})\geq c$, that means, \begin{equation*}K_{1}\left(\int_{M}u_{0}^{a}\right)^{\frac{\zeta}{\beta+1}}\left(\int_{M}\left\vert \nabla  u_{0}^{q } \right\vert^{p}d\mu\right)^{\frac{\beta}{\beta+1}}\geq\int_{M }u_{0}^{1+q } d\mu.\end{equation*} Setting $v=u_{0}^{q}\in W^{1, p}(M)$ yields \begin{equation*}K_{1}\left(\int_{M}v^{\frac{a}{q}}\right)^{\frac{\zeta}{\beta+1}}\left(\int_{M}\left\vert \nabla  v \right\vert^{p}d\mu\right)^{\frac{\beta}{\beta+1}}\geq\int_{M }v^{\frac{1+q}{q} } d\mu.\end{equation*}
Thus, \begin{equation}\label{GNI}\left(\int_{M }v^{\frac{1+q}{q} } d\mu\right)^{\frac{q}{1+q}}\leq K_{2}\left(\int_{M}v^{\frac{a}{q}}\right)^{\frac{\zeta q}{(\beta+1)(1+q)}}\left(\int_{M}\left\vert \nabla  v \right\vert^{p}d\mu\right)^{\frac{\beta q}{(\beta+1)(1+q)}}.\end{equation}
Note that $$1-\frac{\beta qp}{(1+q)(\beta+1)}=\frac{\zeta a}{(1+q)(\beta+1)}$$ since $$\beta=\frac{a\zeta-(1+q)}{D}.$$
Hence, we can write (\ref{GNI}) as \begin{equation}\label{GNI2}
||v||_{L^{r}(M)}\leq K_{2} ||v||_{L^{s}(M)}^{1-\theta}||\nabla v||_{L^{p}(M)}^{\theta},
\end{equation} where $$r=\frac{1+q}{q}, \quad \theta=\frac{\beta q p}{(1+q)(\beta+1)}\quad \textnormal{and}\quad s=\frac{a}{q}.$$ From Theorem 3.1 in \cite{bakry1995sobolev} it therefore follows that \begin{equation}\label{sobo}||v||_{L^{\omega}(M)}\leq C||\nabla v||_{L^{p}(M)},
\end{equation} with $\omega$ given by $$\frac{\theta}{\omega}=\frac{1}{r}-\frac{1-\theta}{s}.$$
Substituting $r, \theta$ and $s$ in this equation yields $$\omega=\frac{p}{1-\nu}=p\kappa,$$ which implies the Sobolev inequality (\ref{weightsobo}) and finishes the proof.
\end{proof}
	
	\begin{da} \normalfont
		This article has no associated data.
	\end{da}
	
	\bibliographystyle{abbrv}
	\bibliography{librarycacc}

\begin{thebibliography}{10}

\bibitem{alt1983quasilinear}
H.~W. Alt and S.~Luckhaus.
\newblock Quasilinear elliptic-parabolic differential equations.
\newblock {\em Math. z}, 183(3):311--341, 1983.

\bibitem{andreucci1990new}
D.~Andreucci and E.~Di~Benedetto.
\newblock A new approach to initial traces in nonlinear filtration.
\newblock In {\em Annales de l'Institut Henri Poincar{\'e} C, Analyse non
  lin{\'e}aire}, volume~7, pages 305--334. Elsevier, 1990.

\bibitem{andreucci2015optimal}
D.~Andreucci and A.~F. Tedeev.
\newblock Optimal decay rate for degenerate parabolic equations on noncompact
  manifolds.
\newblock {\em Methods Appl. Anal}, 22(4):359--376, 2015.

\bibitem{bakry1995sobolev}
D.~Bakry, T.~Coulhon, M.~Ledoux, and L.~Saloff-Coste.
\newblock Sobolev inequalities in disguise.
\newblock {\em Indiana University Mathematics Journal}, pages 1033--1074, 1995.

\bibitem{barbu2025leibenson}
V.~Barbu, S.~Grube, M.~Rehmeier, and M.~R{\"o}ckner.
\newblock The {L}eibenson process.
\newblock {\em arXiv preprint arXiv:2508.12979}, 2025.

\bibitem{barenblatt1952self}
G.~I. Barenblatt.
\newblock On self-similar motions of a compressible fluid in a porous medium.
\newblock {\em Akad. Nauk SSSR. Prikl. Mat. Meh}, 16(6):679--698, 1952.

\bibitem{benilan1995strong}
P.~B{\'e}nilan and R.~Gariepy.
\newblock Strong solutions in {L}1 of degenerate parabolic equations.
\newblock {\em Journal of differential equations}, 119(2):473--502, 1995.

\bibitem{bonforte2006super}
M.~Bonforte and G.~Grillo.
\newblock Super and ultracontractive bounds for doubly nonlinear evolution
  equations.
\newblock {\em Rev. Mat. Iberoamericana}, 22(1):111--129, 2006.

\bibitem{bonforte2007singular}
M.~Bonforte and G.~Grillo.
\newblock Singular evolution on manifolds, their smoothing properties, and
  {S}obolev inequalities.
\newblock {\em Discrete Contin. Dyn. Syst}, 2007:130--137, 2007.

\bibitem{bonforte2008fast}
M.~Bonforte, G.~Grillo, and J.~L. Vazquez.
\newblock Fast diffusion flow on manifolds of nonpositive curvature.
\newblock {\em Journal of Evolution Equations}, 8(1):99--128, 2008.

\bibitem{cipriani2001uniform}
F.~Cipriani and G.~Grillo.
\newblock Uniform bounds for solutions to quasilinear parabolic equations.
\newblock {\em Journal of Differential Equations}, 177(1):209--234, 2001.

\bibitem{coulhon2016regularisation}
T.~Coulhon and D.~Hauer.
\newblock Regularisation effects of nonlinear semigroups.
\newblock {\em arXiv preprint arXiv:1604.08737}, 2016.

\bibitem{de2022wasserstein}
N.~De~Ponti, M.~Muratori, and C.~Orrieri.
\newblock Wasserstein stability of porous medium-type equations on manifolds
  with {R}icci curvature bounded below.
\newblock {\em Journal of Functional Analysis}, 283(9):109661, 2022.

\bibitem{grigor2026sharp}
A.~Grigor'yan, J.~Sun, and P.~S{\"u}rig.
\newblock Sharp long distance upper bounds for solutions of {L}eibenson's
  equation on {R}iemannian manifolds.
\newblock {\em arXiv preprint arXiv:2603.27791}, 2026.

\bibitem{Grigoryan2024}
A.~Grigor'yan and P.~S{\"u}rig.
\newblock Sharp propagation rate for {L}eibenson’s equation on {R}iemannian
  manifolds.
\newblock {\em Ann. Scuola Norm. Super. Pisa}, 2024.

\bibitem{Grigoryan2024a}
A.~Grigor'yan and P.~S{\"u}rig.
\newblock Upper bounds for solutions of {L}eibenson's equation on {R}iemannian
  manifolds.
\newblock {\em Journal of Functional Analysis}, 288(10):110878, 2025.

\bibitem{grigor2006heat}
A.~Grigor’yan.
\newblock Heat kernels on weighted manifolds and applications.
\newblock {\em Cont. Math}, 398(2006):93--191, 2006.

\bibitem{grigor2024finite}
A.~Grigor’yan and P.~S{\"u}rig.
\newblock Finite propagation speed for {L}eibenson’s equation on {R}iemannian
  manifolds.
\newblock {\em Comm. Anal. Geom}, 32(9):2467--2504, 2024.

\bibitem{hoffman1974sobolev}
D.~Hoffman and J.~Spruck.
\newblock Sobolev and isoperimetric inequalities for {R}iemannian submanifolds.
\newblock {\em Communications on Pure and Applied Mathematics}, 27(6):715--727,
  1974.

\bibitem{ishige1996existence}
K.~Ishige.
\newblock On the existence of solutions of the cauchy problem for a doubly
  nonlinear parabolic equation.
\newblock {\em SIAM Journal on Mathematical Analysis}, 27(5):1235--1260, 1996.

\bibitem{ivanov1997regularity}
A.~V. Ivanov.
\newblock Regularity for doubly nonlinear parabolic equations.
\newblock {\em Journal of Mathematical Sciences}, 83(1):22--37, 1997.

\bibitem{leibenzon1945general}
L.~Leibenson.
\newblock General problem of the movement of a compressible fluid in a porous
  medium. izv akad. nauk sssr.
\newblock {\em Geography and Geophysics}, 9:7--10, 1945.

\bibitem{leibenson1945turbulent}
L.~Leibenson.
\newblock Turbulent movement of gas in a porous medium.
\newblock {\em Izv. Akad. Nauk SSSR Ser. Geograf. Geofiz}, 9:3--6, 1945.

\bibitem{meglioli2025global}
G.~Meglioli, F.~Oliva, and F.~Petitta.
\newblock Global existence for a {L}eibenson type equation with reaction on
  {R}iemannian manifolds.
\newblock {\em Nonlinear Analysis}, 263:113967, 2026.

\bibitem{raviart1970resolution}
P.-A. Raviart.
\newblock Sur la r{\'e}solution de certaines {\'e}quations paraboliques non
  lin{\'e}aires.
\newblock {\em Journal of Functional Analysis}, 5(2):299--328, 1970.

\bibitem{surig2024finite}
P.~S{\"u}rig.
\newblock Finite extinction time for subsolutions of the weighted {L}eibenson
  equation on {R}iemannian manifolds.
\newblock {\em arXiv preprint arXiv:2412.06496}, 2024.

\bibitem{surig2024sharp}
P.~S{\"u}rig.
\newblock Sharp sub-{G}aussian upper bounds for subsolutions of {T}rudinger’s
  equation on {R}iemannian manifolds.
\newblock {\em Nonlinear Analysis}, 249:113641, 2024.

\bibitem{surig2025gradient}
P.~S{\"u}rig.
\newblock Gradient estimates for {L}eibenson's equation on {R}iemannian
  manifolds.
\newblock {\em arXiv preprint arXiv:2506.07221}, 2025.

\bibitem{surig2026existence}
P.~S{\"u}rig.
\newblock Existence results for {L}eibenson's equation on {R}iemannian
  manifolds.
\newblock {\em arXiv preprint arXiv:2601.20640}, 2026.

\bibitem{tsutsumi1988solutions}
M.~Tsutsumi.
\newblock On solutions of some doubly nonlinear degenerate parabolic equations
  with absorption.
\newblock {\em Journal of mathematical analysis and applications},
  132(1):187--212, 1988.

\bibitem{varopoulos1985hardy}
N.~T. Varopoulos.
\newblock Hardy-{L}ittlewood theory for semigroups.
\newblock {\em Journal of functional analysis}, 63(2):240--260, 1985.

\end{thebibliography}
	
	\emph{Universit\"{a}t Bielefeld, Fakult\"{a}t f\"{u}r Mathematik, Postfach
		100131, D-33501, Bielefeld, Germany}
	
	\texttt{philipp.suerig@uni-bielefeld.de}
\end{document}